\title{Finding a Unique Solution to Radon-Kaczmarz Puzzles}
\author{Steven Rossi \thanks{Utica College} \and Xiao Xiao \thanks{Utica College}}
\documentclass{article}
\usepackage{amsthm,amsmath,amsfonts,mathrsfs,amssymb, eucal,comment, enumerate,graphicx,pgf,tikz}
\numberwithin{equation}{section}

\newtheorem{theorem}{Theorem}
\newtheorem{lemma}{Lemma}
\newtheorem{proposition}{Proposition}
\newtheorem{definition}{Definition}
\newtheorem{exploration}{Exploration}
\newtheorem{remark}{Remark}
\newtheorem{corollary}{Corollary}
\renewenvironment{proof}{{\sc Proof:}}{~\hfill QED}

\begin{document}
\newpage
\maketitle
\begin{abstract}
Solving a Radon-Kaczmarz puzzle involves filling a square grid with positive integers, each between one and nine, satisfying certain clues coming from the sum of entries that lie on the same line in the square grid. Given a set of slopes (of a particular order) that define clues of Radon-Kaczmarz puzzles, we give an upper bound of the size such that any solvable Radon-Kaczmarz puzzle whose size is less than or equal to that is uniquely solvable.
\end{abstract}




\section{Introduction}


A $3 \times 3$ Radon-Kaczmarz puzzle (or RK puzzle), such as the one in Figure \ref{RKPuzzleBasicExample}, involves filling in an empty $3 \times 3$ square grid with positive integers, each between one and nine, so that the sum of each row, column, and diagonal is equal to its own aggregate given in the four grids of Figure \ref{RKPuzzleBasicExample}. 


\begin{exploration}
How many solution(s) can you find for the RK puzzle described in Figure \ref{RKPuzzleBasicExample} and how do you know there are not any more?
\end{exploration}

\begin{center}
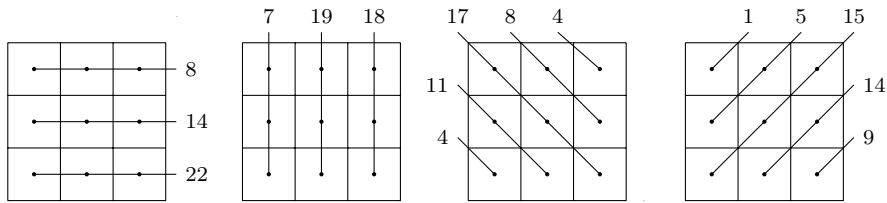
\begin{figure}[!htbp]
\centering
\begin{tikzpicture}[scale=0.7]
\centering
\draw[thin,color=black,step=1cm] (0,0) grid (3,3);
\filldraw[fill=black] (0.5,0.5) circle (0.3mm);
\filldraw[fill=black] (1.5,0.5) circle (0.3mm);
\filldraw[fill=black] (2.5,0.5) circle (0.3mm);
\filldraw[fill=black] (0.5,1.5) circle (0.3mm);
\filldraw[fill=black] (1.5,1.5) circle (0.3mm);
\filldraw[fill=black] (2.5,1.5) circle (0.3mm);
\filldraw[fill=black] (0.5,2.5) circle (0.3mm);
\filldraw[fill=black] (1.5,2.5) circle (0.3mm);
\filldraw[fill=black] (2.5,2.5) circle (0.3mm);
\filldraw[fill=black] (3.2,3.5) circle (0.0mm);
\draw[thin, color=black] (0.5,0.5) -- (3.2,0.5) node[right] {{\footnotesize $22$}};
\draw[thin, color=black] (0.5,1.5) -- (3.2,1.5) node[right] {{\footnotesize $14$}};
\draw[thin, color=black] (0.5,2.5) -- (3.2,2.5) node[right] {{\footnotesize $8$}};
\end{tikzpicture} \;
\begin{tikzpicture}[scale=0.7]
\draw[thin,color=black,step=1cm] (0,0) grid (3,3);
\filldraw[fill=black] (0.5,0.5) circle (0.3mm);
\filldraw[fill=black] (1.5,0.5) circle (0.3mm);
\filldraw[fill=black] (2.5,0.5) circle (0.3mm);
\filldraw[fill=black] (0.5,1.5) circle (0.3mm);
\filldraw[fill=black] (1.5,1.5) circle (0.3mm);
\filldraw[fill=black] (2.5,1.5) circle (0.3mm);
\filldraw[fill=black] (0.5,2.5) circle (0.3mm);
\filldraw[fill=black] (1.5,2.5) circle (0.3mm);
\filldraw[fill=black] (2.5,2.5) circle (0.3mm);
\draw[thin, color=black] (0.5,0.5) -- (0.5,3.2) node[above] {{\footnotesize $7$}};
\draw[thin, color=black] (1.5,0.5) -- (1.5,3.2) node[above] {{\footnotesize $19$}};
\draw[thin, color=black] (2.5,0.5) -- (2.5,3.2) node[above] {{\footnotesize $18$}};
\end{tikzpicture}\;
\begin{tikzpicture}[scale=0.7]
\draw[thin,color=black,step=1cm] (0,0) grid (3,3);
\filldraw[fill=black] (0.5,0.5) circle (0.3mm);
\filldraw[fill=black] (1.5,0.5) circle (0.3mm);
\filldraw[fill=black] (2.5,0.5) circle (0.3mm);
\filldraw[fill=black] (0.5,1.5) circle (0.3mm);
\filldraw[fill=black] (1.5,1.5) circle (0.3mm);
\filldraw[fill=black] (2.5,1.5) circle (0.3mm);
\filldraw[fill=black] (0.5,2.5) circle (0.3mm);
\filldraw[fill=black] (1.5,2.5) circle (0.3mm);
\filldraw[fill=black] (2.5,2.5) circle (0.3mm);
\draw[thin, color=black] (2.5,2.5) -- (1.7,3.2) node[above] {{\footnotesize $4$}};
\draw[thin, color=black] (2.5,1.5) -- (0.8,3.2) node[above] {{\footnotesize $8$}};
\draw[thin, color=black] (2.5,0.5) -- (-0.2,3.2) node[above] {{\footnotesize $17$}};
\draw[thin, color=black] (1.5,0.5) -- (-0.2,2.2) node[left] {{\footnotesize $11$}};
\draw[thin, color=black] (0.5,0.5) -- (-0.2,1.2) node[left] {{\footnotesize $4$}};
\end{tikzpicture}\;
\begin{tikzpicture}[scale=0.7]
\draw[thin,color=black,step=1cm] (0,0) grid (3,3);
\filldraw[fill=black] (0.5,0.5) circle (0.3mm);
\filldraw[fill=black] (1.5,0.5) circle (0.3mm);
\filldraw[fill=black] (2.5,0.5) circle (0.3mm);
\filldraw[fill=black] (0.5,1.5) circle (0.3mm);
\filldraw[fill=black] (1.5,1.5) circle (0.3mm);
\filldraw[fill=black] (2.5,1.5) circle (0.3mm);
\filldraw[fill=black] (0.5,2.5) circle (0.3mm);
\filldraw[fill=black] (1.5,2.5) circle (0.3mm);
\filldraw[fill=black] (2.5,2.5) circle (0.3mm);
\filldraw[fill=black] (-0.8,0) circle (0.0mm);
\draw[thin, color=black] (0.5,2.5) -- (1.2,3.2) node[above] {{\footnotesize $1$}};
\draw[thin, color=black] (0.5,1.5) -- (2.2,3.2) node[above] {{\footnotesize $5$}};
\draw[thin, color=black] (0.5,0.5) -- (3.2,3.2) node[above] {{\footnotesize $15$}};
\draw[thin, color=black] (1.5,0.5) -- (3.2,2.2) node[right] {{\footnotesize $14$}};
\draw[thin, color=black] (2.5,0.5) -- (3.2,1.2) node[right] {{\footnotesize $9$}};
\end{tikzpicture}
\caption{Example of a $3 \times 3$  Puzzle} \label{RKPuzzleBasicExample}
\end{figure}
\end{center}

We say that the clues in the first (resp. second, third and fourth) grid of Figure \ref{RKPuzzleBasicExample} are given by lines of slope $0$ (resp. $\infty$, $-1$ and $1$) because each aggregate is obtained by summing the grid entries whose centers lie on a line of slope $0$ (resp. $\infty$, $-1$ and $1$). It is not hard to check that every solvable $3 \times 3$ RK puzzle has a unique solution given clues coming from $0,\infty,-1$ and $1$. 

\begin{exploration}
How many solution(s) can you find for the RK puzzle described in Figure \ref{RKPuzzleBasicExampleNonUnique}?
\end{exploration}

\begin{center}
\begin{figure}[!htbp]
\centering
\begin{tikzpicture}[scale=0.7]
\centering
\draw[thin,color=black,step=.75cm] (0,0) grid (3,3);
\filldraw[fill=black] (0.375,0.375) circle (0.3mm);
\filldraw[fill=black] (1.125,0.375) circle (0.3mm);
\filldraw[fill=black] (1.875,0.375) circle (0.3mm);
\filldraw[fill=black] (2.625,0.375) circle (0.3mm);
\filldraw[fill=black] (0.375,1.125) circle (0.3mm);
\filldraw[fill=black] (1.125,1.125) circle (0.3mm);
\filldraw[fill=black] (1.875,1.125) circle (0.3mm);
\filldraw[fill=black] (2.625,1.125) circle (0.3mm);
\filldraw[fill=black] (0.375,1.875) circle (0.3mm);
\filldraw[fill=black] (1.125,1.875) circle (0.3mm);
\filldraw[fill=black] (1.875,1.875) circle (0.3mm);
\filldraw[fill=black] (2.625,1.875) circle (0.3mm);
\filldraw[fill=black] (0.375,2.625) circle (0.3mm);
\filldraw[fill=black] (1.125,2.625) circle (0.3mm);
\filldraw[fill=black] (1.875,2.625) circle (0.3mm);
\filldraw[fill=black] (2.625,2.625) circle (0.3mm);

\filldraw[fill=black] (3.2,3.5) circle (0.0mm);

\draw[thin, color=black] (0.375,0.375) -- (3.2,0.375) node[right] {{\footnotesize $15$}};
\draw[thin, color=black] (0.375,1.125) -- (3.2,1.125) node[right] {{\footnotesize $16$}};
\draw[thin, color=black] (0.375,1.875) -- (3.2,1.875) node[right] {{\footnotesize $11$}};
\draw[thin, color=black] (0.375,2.625) -- (3.2,2.625) node[right] {{\footnotesize $20$}};
\end{tikzpicture}\;
\begin{tikzpicture}[scale=0.7]
\draw[thin,color=black,step=.75cm] (0,0) grid (3,3);
\filldraw[fill=black] (0.375,0.375) circle (0.3mm);
\filldraw[fill=black] (1.125,0.375) circle (0.3mm);
\filldraw[fill=black] (1.875,0.375) circle (0.3mm);
\filldraw[fill=black] (2.625,0.375) circle (0.3mm);
\filldraw[fill=black] (0.375,1.125) circle (0.3mm);
\filldraw[fill=black] (1.125,1.125) circle (0.3mm);
\filldraw[fill=black] (1.875,1.125) circle (0.3mm);
\filldraw[fill=black] (2.625,1.125) circle (0.3mm);
\filldraw[fill=black] (0.375,1.875) circle (0.3mm);
\filldraw[fill=black] (1.125,1.875) circle (0.3mm);
\filldraw[fill=black] (1.875,1.875) circle (0.3mm);
\filldraw[fill=black] (2.625,1.875) circle (0.3mm);
\filldraw[fill=black] (0.375,2.625) circle (0.3mm);
\filldraw[fill=black] (1.125,2.625) circle (0.3mm);
\filldraw[fill=black] (1.875,2.625) circle (0.3mm);
\filldraw[fill=black] (2.625,2.625) circle (0.3mm);

\filldraw[fill=black] (3.2,3.5) circle (0.0mm);

\draw[thin, color=black] (0.375,0.375) -- (0.375,3.2) node[above] {{\footnotesize $16$}};
\draw[thin, color=black] (1.125,0.375) -- (1.125,3.2) node[above] {{\footnotesize $16$}};
\draw[thin, color=black] (1.875,0.375) -- (1.875,3.2) node[above] {{\footnotesize $15$}};
\draw[thin, color=black] (2.625,0.375) -- (2.625,3.2) node[above] {{\footnotesize $15$}};
\end{tikzpicture}\;
\begin{tikzpicture}[scale=0.7]
\draw[thin,color=black,step=.75cm] (0,0) grid (3,3);
\filldraw[fill=black] (0.375,0.375) circle (0.3mm);
\filldraw[fill=black] (1.125,0.375) circle (0.3mm);
\filldraw[fill=black] (1.875,0.375) circle (0.3mm);
\filldraw[fill=black] (2.625,0.375) circle (0.3mm);
\filldraw[fill=black] (0.375,1.125) circle (0.3mm);
\filldraw[fill=black] (1.125,1.125) circle (0.3mm);
\filldraw[fill=black] (1.875,1.125) circle (0.3mm);
\filldraw[fill=black] (2.625,1.125) circle (0.3mm);
\filldraw[fill=black] (0.375,1.875) circle (0.3mm);
\filldraw[fill=black] (1.125,1.875) circle (0.3mm);
\filldraw[fill=black] (1.875,1.875) circle (0.3mm);
\filldraw[fill=black] (2.625,1.875) circle (0.3mm);
\filldraw[fill=black] (0.375,2.625) circle (0.3mm);
\filldraw[fill=black] (1.125,2.625) circle (0.3mm);
\filldraw[fill=black] (1.875,2.625) circle (0.3mm);
\filldraw[fill=black] (2.625,2.625) circle (0.3mm);

\filldraw[fill=black] (3.2,-.2) circle (0.0mm);

\draw[thin, color=black] (2.625,2.625) -- (2.025,3.225) node[above] {{\footnotesize $2$}};
\draw[thin, color=black] (2.625,1.875) -- (1.275,3.225) node[above] {{\footnotesize $7$}};
\draw[thin, color=black] (2.625,1.125) -- (0.525,3.225) node[above] {{\footnotesize $15$}};
\draw[thin, color=black] (2.625,0.375) -- (-0.225,3.225) node[above] {{\footnotesize $19$}};
\draw[thin, color=black] (1.875,0.375) -- (-0.2,2.45) node[left] {{\footnotesize $11$}};
\draw[thin, color=black] (1.125,0.375) -- (-0.2,1.7) node[left] {{\footnotesize $3$}};
\draw[thin, color=black] (0.375,0.375) -- (-0.2,0.95) node[left] {{\footnotesize $5$}};
\end{tikzpicture}\;
\begin{tikzpicture}[scale=0.7]
\draw[thin,color=black,step=.75cm] (0,0) grid (3,3);
\filldraw[fill=black] (0.375,0.375) circle (0.3mm);
\filldraw[fill=black] (1.125,0.375) circle (0.3mm);
\filldraw[fill=black] (1.875,0.375) circle (0.3mm);
\filldraw[fill=black] (2.625,0.375) circle (0.3mm);
\filldraw[fill=black] (0.375,1.125) circle (0.3mm);
\filldraw[fill=black] (1.125,1.125) circle (0.3mm);
\filldraw[fill=black] (1.875,1.125) circle (0.3mm);
\filldraw[fill=black] (2.625,1.125) circle (0.3mm);
\filldraw[fill=black] (0.375,1.875) circle (0.3mm);
\filldraw[fill=black] (1.125,1.875) circle (0.3mm);
\filldraw[fill=black] (1.875,1.875) circle (0.3mm);
\filldraw[fill=black] (2.625,1.875) circle (0.3mm);
\filldraw[fill=black] (0.375,2.625) circle (0.3mm);
\filldraw[fill=black] (1.125,2.625) circle (0.3mm);
\filldraw[fill=black] (1.875,2.625) circle (0.3mm);
\filldraw[fill=black] (2.625,2.625) circle (0.3mm);

\filldraw[fill=black] (-0.75,3.5) circle (0.0mm);

\draw[thin, color=black] (1.125,0.375) -- (3.225,2.475) node[right] {{\footnotesize $11$}};
\draw[thin, color=black] (1.875,0.375) -- (3.225,1.725) node[right] {{\footnotesize $11$}};
\draw[thin, color=black] (2.625,0.375) -- (3.225,0.975) node[right] {{\footnotesize $3$}};
\draw[thin, color=black] (0.375,0.375) -- (3.225,3.225) node[right] {{\footnotesize $13$}};
\draw[thin, color=black] (0.375,1.125) -- (2.475,3.225) node[above] {{\footnotesize $7$}};
\draw[thin, color=black] (0.375,1.875) -- (1.725,3.225) node[above] {{\footnotesize $9$}};
\draw[thin, color=black] (0.375,2.625) -- (0.975,3.225) node[above] {{\footnotesize $8$}};
\end{tikzpicture}
\caption{Example of a $4 \times 4$ RK puzzle} \label{RKPuzzleBasicExampleNonUnique}
\end{figure}
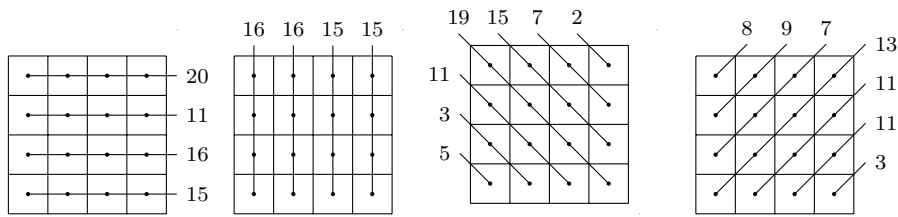
\end{center}

If we add another set of clues coming from lines of slope $-1/2$ in Figure \ref{RKPuzzleBasicExampleUniqueAgain}, the solution of the $4 \times 4$ RK puzzle defined by the clues in Figure \ref{RKPuzzleBasicExampleNonUnique} (together with the clues in Figure \ref{RKPuzzleBasicExampleUniqueAgain} is unique. In fact, one can check that every solvable $4 \times 4$ RK puzzle has a unique solution if we have all the clues coming from lines of slope $0$, $\infty$, $-1$, $1$ and $-1/2$.

\begin{center}
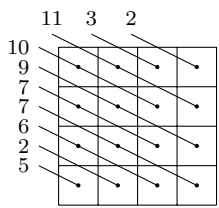
\begin{figure}[!htbp]
\centering
\begin{tikzpicture}[scale=0.7]
\draw[thin,color=black,step=.75cm] (0,0) grid (3,3);
\filldraw[fill=black] (0.375,0.375) circle (0.3mm);
\filldraw[fill=black] (1.125,0.375) circle (0.3mm);
\filldraw[fill=black] (1.875,0.375) circle (0.3mm);
\filldraw[fill=black] (2.625,0.375) circle (0.3mm);
\filldraw[fill=black] (0.375,1.125) circle (0.3mm);
\filldraw[fill=black] (1.125,1.125) circle (0.3mm);
\filldraw[fill=black] (1.875,1.125) circle (0.3mm);
\filldraw[fill=black] (2.625,1.125) circle (0.3mm);
\filldraw[fill=black] (0.375,1.875) circle (0.3mm);
\filldraw[fill=black] (1.125,1.875) circle (0.3mm);
\filldraw[fill=black] (1.875,1.875) circle (0.3mm);
\filldraw[fill=black] (2.625,1.875) circle (0.3mm);
\filldraw[fill=black] (0.375,2.625) circle (0.3mm);
\filldraw[fill=black] (1.125,2.625) circle (0.3mm);
\filldraw[fill=black] (1.875,2.625) circle (0.3mm);
\filldraw[fill=black] (2.625,2.625) circle (0.3mm);

\filldraw[fill=black] (3.2,-.2) circle (0.0mm);

\draw[thin, color=black] (0.375,0.375) -- (-0.375,0.75) node[left] {{\footnotesize $5$}};
\draw[thin, color=black] (1.125,0.375) -- (-0.375,1.125) node[left] {{\footnotesize $2$}};
\draw[thin, color=black] (1.875,0.375) -- (-0.375,1.5) node[left] {{\footnotesize $6$}};
\draw[thin, color=black] (2.625,0.375) -- (-0.375,1.875) node[left] {{\footnotesize $7$}};
\draw[thin, color=black] (1.875,1.125) -- (-0.375,2.25) node[left] {{\footnotesize $7$}};
\draw[thin, color=black] (2.625,1.125) -- (-0.375,2.625) node[left] {{\footnotesize $9$}};
\draw[thin, color=black] (1.875,1.875) -- (-0.375,3) node[left] {{\footnotesize $10$}};
\draw[thin, color=black] (2.625,1.875) -- (-0.125,3.25) node[above] {{\footnotesize $11$}};
\draw[thin, color=black] (1.875,2.625) -- (0.625,3.25) node[above] {{\footnotesize $3$}};
\draw[thin, color=black] (2.625,2.625) -- (1.375,3.25) node[above] {{\footnotesize $2$}};
\end{tikzpicture}
\caption{Additional Clues of Slope $-1/2$ for the $4 \times 4$ RK puzzle} \label{RKPuzzleBasicExampleUniqueAgain}
\end{figure}
\end{center}

This leads us to ask the following question: given a solvable RK puzzle of any size, what is the relationship between the size and the clue slopes of the puzzle so that the puzzle is uniquely solvable? For the sake of generality, we will work with RK puzzles whose clues and solutions are real numbers in this paper. 

\begin{theorem} \label{theorem:maintheorem}
For every positive integer $n$, given any solvable $n \times n$ RK puzzle with clues coming from the set \[\{0, \infty, \pm 1, \pm 2, \pm \frac{1}{2}, \dots, \pm q, \pm \frac{1}{q}\}.\] If $n \leq q^2+3q-1$, then the RK puzzle has a unique solution.
\end{theorem}
Theorem \ref{theorem:maintheorem} is clearly true when $q=1$. See Theorem \ref{theorem:maintheoremgivenclues} Part (4) for the proof of Theorem \ref{theorem:maintheorem} when $q \geq 2$. 

The authors would like to thank Julian Fleron for creating and introducing the RK puzzle to them in \cite{DAoM:GP}. Interested readers are encouraged to work out the task sequence in \cite[Chapter 3.2]{DAoM:GP}. Furthermore, Fleron wrote a paper \cite{Fleron:RKHorizons} that explains the connections between solving RK puzzles and CAT scan, and an integer linear programming approach to solve RK puzzles. 

\section{Basic Terminology and Facts}
In this section, we give a formal definition of an $n \times m$ RK puzzle and prove several basic facts about the uniqueness of its solution.

For every positive integer $n$, let $I_n = \{1,2,\dots,n\}$. For every pair of positive integers $(n,m) \in \mathbb{Z}_{>0}^2$, let $I_{n,m} := I_n \times I_m \subset \mathbb{R}^2$ be a lattice of points with integral coordinates on the Cartesian plane with the natural embedding. Let $B_{n,m} := (\{1,n\} \times I_m) \cup (I_n \times \{1,m\})$ be the border of $I_{n,m}$. For each $s \in \mathbb{Q} \cup \{\infty\}$, let $\mathscr{L}_s$ be the set of all lines $\ell$ in $\mathbb{R}^2$ of slope $s$ such that $\ell \cap I_{n,m} \neq \emptyset$. For each $\ell \in \mathscr{L}_s$, we construct a linear clue polynomial with respect to $I_{n,m}$ of $nm$ variables:
\[P_{\ell,s}(X) = \sum_{i=1}^n \sum_{j=1}^m a_{i,j}X_{i,j}\]
where 
\begin{equation*}
a_{i,j}=
\begin{cases}
1 & \text{if} \quad (i,j) \in \ell \cap I_{n,m},\\
0 & \text{otherwise}.
\end{cases}
\end{equation*}
\begin{definition}
Let $T$ be a finite subset of $\mathbb{Q} \cup \{\infty\}$. An $n \times m$ RK puzzle of the slope set $T$ is a pair $\mathscr{M} = (I_{n,m}, \mathcal{C}_T)$ where $\mathcal{C}_T = \cup_{t \in T} \mathcal{C}^*_t$ is called the set of clues of $\mathscr{M}$ and $\mathcal{C}^*_t$ consists of all the equations of the type $P_{\ell,t}(X)=c_{\ell,t}$ where $\ell \in \mathscr{L}_t$ and $c_{\ell,t} \in \mathbb{R}$.
\end{definition}
Clearly, the set of clues $\mathcal{C}_T$ of any RK puzzle form a system of linear equations. A solution of an RK puzzle is a solution of the system of linear equations defined by $\mathcal{C}_T$. Note that here we allow the constant term of each clue to be any real number instead of just positive integers, and the solution of RK puzzle to be any real numbers instead of just positive integers between one and nine.

We now study the symmetries of the uniqueness of solutions of RK puzzles. This will allow us to focus our study on particular areas of the puzzles and then apply those results elsewhere.

\begin{proposition} \label{proposition:symmetry180}
Let $\mathscr{M} = (I_{n,m}, \mathcal{C}_T)$ be a solvable $n \times m$ RK puzzle. If the $(i,j)$-entry of the RK puzzle is uniquely solvable, then the $(n-i+1,m-j+1)$-entry is also uniquely solvable.
\end{proposition}
\begin{proof}
We construct an RK puzzle $\mathscr{M}'$ such that if $f: \{X_{i,j}\}_{(i,j) \in I_{n,m}} \to \mathbb{R}$ is a solution of $\mathscr{M}$, then the composition of
\begin{align*}
\mu: \{X_{i,j}\}_{(i,j) \in I_{n,m}} &\to \{X_{i,j}\}_{(i,j) \in I_{n,m}} \\
X_{i,j} &\mapsto X_{n-i+1,m-j+1}
\end{align*}
and $f$ is a solution of $\mathscr{M}'$.

Let $P_{\ell,p/q}(X)$ be a clue polynomial in $\mathcal{C}_T$ with $\ell : px-qy+b=0$, then $P_{\ell,p/q}(\mu(X)) = P_{\ell',p/q}(X)$ where $\ell' : px-qy-p(n+1)+q(m+1)+b=0$ is another clue polynomial of $\mathcal{C}_T$. For each clue $P_{\ell,p/q}(X) = c_{\ell,p/q}$ in $\mathcal{C}_T$, we define a clue $P_{\ell,p/q}(\mu(X)) = c_{\ell,p/q}$ in $\mathcal{C}'_T$ of $\mathscr{M}'$. The systems of linear equations of $\mathcal{C}_T$ and $\mathcal{C}'_T$ have the same coefficient matrix (up to a reordering of clue polynomials) but different augmented matrices. As the uniqueness of the solution of systems of linear equations only depends on the coefficient matrix, we know that if the $(i,j)$-entry of $\mathscr{M}$ is uniquely solvable, then the $(i,j)$-entry of $\mathscr{M}'$ is also uniquely solvable. But the solution of the $(i,j)$-entry of $\mathscr{M}'$ is the solution of the $(n-i+1,m-j+1)$-entry of $\mathscr{M}$, so the solution of the $(n-i+1,m-j+1)$-entry of $\mathscr{M}$ is unique.
\end{proof}

\begin{proposition} \label{proposition:symmetryvh}
Let $\mathscr{M} = (I_{n,m}, \mathcal{C}_T)$ be a solvable $n \times m$ RK puzzle with slope set $T$ such that if $t \in T$ then $-t \in T$. If the $(i,j)$-entry of the RK puzzle is uniquely solvable, then the $(n-i+1,j)$-entry and the $(i,m-j+1)$-entry are also uniquely solvable.
\end{proposition}

\begin{proof}
The fact that the $(n-i+1,j)$-entry is uniquely solvable can be proved using the same argument as Proposition \ref{proposition:symmetry180} by letting
\begin{align*}
\mu: \{X_{i,j}\}_{(i,j) \in I_{n,m}} &\to \{X_{i,j}\}_{(i,j) \in I_{n,m}} \\
X_{i,j} &\mapsto X_{n-i+1,j}
\end{align*}
By Proposition \ref{proposition:symmetry180}, we know the $(i,m-j+1)$-entry is also uniquely solvable as well.
\end{proof}

\begin{proposition} \label{proposition:symmetrypair}
Let $\mathscr{M} = (I_{n,n}, \mathcal{C}_T)$ be a solvable $n \times n$ RK puzzle with slope set $T$ such that if $t \in T$ then $1/t \in T$. If the $(i,j)$-entry of the RK puzzle is uniquely solvable, then the $(j,i)$-entry and the $(n-j+1,n-i+1)$-entry are uniquely solvable.
\end{proposition}

\begin{proof}
The fact that the $(j,i)$-entry is uniquely solvable can be proved using the same argument as Proposition \ref{proposition:symmetry180} by letting
\begin{align*}
\mu: \{X_{i,j}\}_{i,j \in I_n} &\to \{X_{i,j}\}_{i,j \in I_n} \\
X_{i,j} &\mapsto X_{j,i}
\end{align*}
By Proposition \ref{proposition:symmetry180}, we know that the $(n-j+1,n-i+1)$-entry is uniquely solvable as well.
\end{proof}

\section{Invariants}

Let $S := \mathbb{Z} \cup (1/\mathbb{Z}) \subset \mathbb{Q} \cup \{\infty\}$ where $1/0 := \infty$. We define the following order $\prec$ on $S$:
\[0 \prec \infty \prec -1 \prec 1 \prec -\frac{1}{2} \prec -2 \prec \frac{1}{2} \prec 2 \prec -\frac{1}{3} \prec -3 \prec \frac{1}{3} \prec 3 \prec \cdots.\]
For the next two sections, the set of slopes $T$, which defines the clue set $\mathcal{C}_T$, will always be a subset of $S$ (instead of $\mathbb{Q} \cup \{\infty\}$) that contains all slopes from $0$ to $s$ (according to the order $\prec$) for some $s \in S$. In this case, we denote $\mathcal{C}_T$ by $\mathcal{C}_s$ for brevity. Let $\textrm{RK}_{n,m}^{s}$ be the family of all solvable RK puzzles of the type $\mathscr{M} = (I_{n,m},\mathcal{C}_s)$.


\begin{definition} \label{definition:invariants}
We define the following five invariants of RK puzzles in $\textrm{RK}_{n,m}^{s}$.
\begin{enumerate}
\item Let $k_{n,m}$ (resp. $b_{n,m}$) be the smallest element in $S$ (with respect to the order $\prec$) such that every RK puzzle (resp. the border of every RK puzzle) in $\textrm{RK}_{n,m}^{k_{n,m}}$ (resp. $\textrm{RK}_{n,m}^{b_{n,m}}$) has a unique solution.
\item Let $r_{n,m}$ (resp. $c_{n,m}$) be the smallest element in $S$ (with respect to the order $\prec$) such that either the first or the last row (resp. column) of every RK puzzle in $\textrm{RK}_{n,m}^{r_{n,m}}$ (resp. $\textrm{RK}_{n,m}^{c_{n,m}}$) has a unique solution. Let $s_{n,m} := \textrm{min}\{r_{n,m}, c_{n,m}\}$.
\end{enumerate}
\end{definition}

\begin{remark} \label{remark:symmetry}
By Proposition \ref{proposition:symmetry180}, we can change ``either $\dots$ or'' in Definition \ref{definition:invariants} (2) to ``both $\dots$ and''. Hence $b_{n,m} = \textrm{max}\{r_{n,m},c_{n,m}\}$.
\end{remark}

\begin{proposition}
For every positive integers $n$ and $m$, we have
\begin{enumerate}[(1)]
\item $r_{n,m} \preceq r_{n+1,m}$ and $r_{n,m} \preceq r_{n,m+1}$.
\item $c_{n,m} \preceq c_{n+1,m}$ and $c_{n,m} \preceq c_{n,m+1}$.
\item $s_{n,m} \preceq s_{n+1,m}$ and $s_{n,m} \preceq s_{n,m+1}$.
\end{enumerate}
\end{proposition}

\begin{proof}
We first prove $r_{n,m} \preceq r_{n+1,m}$ in $(1)$. Let $\mathscr{M} = (I_{n,m}, \mathcal{C}_{r_{n+1,m}})$ be in $\textrm{RK}_{n,m}^{r_{n+1,m}}$. We want to show that either the top row or the bottom row of $\mathscr{M}$ has a unique solution.

We construct an $(n+1) \times m$ RK puzzle $\mathscr{M}' = (I_{n+1,m}, \mathcal{C}'_{r_{n+1,m}})$. The lattice $I_{n+1,m}$ has an additional column on the right compared to $I_{n,m}$. Let $P'$ be a clue polynomial with respect to $I_{n+1,m}$ and let $c'_{P'}$ be the number of variables of the type $X_{n+1,i}$ in $P'$. By letting all the variables $X_{n+1,i} = 0$ for all $i$ in $P'$, we get a clue polynomial $P$ with respect to $I_{n,m}$. Let $P = c_P$ be the corresponding clue in $\mathcal{C}_{r_{n+1,m}}$. The clue set $\mathcal{C}'_{r_{n+1,m}}$ consists of clues of the type $P' =c_P+c'_{P'}$ where $P'$ varies among all clue polynomials with respect to $I_{n+1,m}$. A solution of $\mathscr{M}$ together with all $1$'s in the rightmost column gives a solution of $\mathscr{M}'$. By definition of $r_{n+1,m}$, we know that either the top or the bottom row of $\mathscr{M}'$ is uniquely solvable and hence the same is true for $\mathscr{M}$. Thus $r_{n,m} \preceq r_{n+1,m}$. In a similar way, we can prove that $r_{n,m} \preceq r_{n,m+1}$ in (1) and the two inequalities in (2). (3) is direct consequence of (1) and (2) by definition.
\end{proof}

\begin{theorem}
We have $s_{n,m} = k_{n,m}$.
\end{theorem}

\begin{proof}
It is clear from the definition that $s_{n,m} \preceq k_{n,m}$ so it remains to prove that $k_{n,m} \preceq s_{n,m}$.

Let $\mathscr{M} = (I_{n,m}, \mathcal{C}_{s_{n,m}})$ be in $\textrm{RK}_{n,m}^{s_{n,m}}$. By definition of $s_{n,m}$ and Remark \ref{remark:symmetry}, either the first and last rows of $\mathscr{M}$ are uniquely solvable, or the first and last columns of $\mathscr{M}$ are uniquely solvable. We discuss the two cases separately.
\begin{enumerate}
\item Suppose the first and last rows of $\mathscr{M}$ is solvable. By plugging in the solutions of the first row and last row of $\mathscr{M}$ into the clues in $\mathcal{C}_{s_{n,m}}$, we get a new clue set $\mathcal{C}'_{s_{n,m}}$ for an $n \times (m-2)$ RK puzzle $\mathscr{M}' = (I_{n,m-2},\mathcal{C}'_{s_{n,m}})$.
\item Suppose the first and last columns of $\mathscr{M}$ is solvable. In a similar way as (1), we get an $(n-2) \times m$ RK puzzle $\mathscr{M}'' = (I_{n-2,m},\mathcal{C}''_{s_{n,m}})$.
\end{enumerate}
As $s_{n,m-2} \preceq s_{n,m}$ and $s_{n-2,m} \preceq s_{n,m}$, we know that $\mathscr{M}$ is uniquely solvable by induction.
\end{proof}

\begin{corollary} \label{corollary:allequal}
For every positive integers $n$ and $m$, we have $r_{n,m} = c_{n,m} = k_{n,m} = b_{n,m} = s_{n,m}$.
\end{corollary}

\section{Proof of Main Theorem}
In this section, we will only work with $n \times n$ RK puzzles. To ease notation let $\textrm{RK}_{n}^{s} := \textrm{RK}_{n,n}^{s}$ and $k_n := k_{n,n}$.

For any finite sequence $\omega = (w_1, w_2, \dots, w_m)$ of integers and any integers $1 \leq i \leq j \leq m$, let $\omega{(i,j)} := \sum_{k=i}^j w_k$.

\begin{lemma} \label{lemma:staircasereductionbasic}
Let $\mathscr{M} = (I_{n,n}, \mathcal{C}_T)$ be an $n \times n$ RK puzzle. Let $\omega = (w_1, w_2, \dots, w_m)$ be a sequence of non-negative integers such that $\omega(1,m) \leq n$ and all entries in $J_{\omega} = \bigcup_{i=1}^m I_{\omega(i,m), i}$ are uniquely solvable. For any $1 \leq j \leq m+1$, if there exists a clue slope $-1/q \in T$ such that
\[tq \geq \omega(j,j+t-1)+1, \; \textrm{for} \; \;  1 \leq t \leq m-j+1 \tag{*}\]
and
\[tq+1 \leq \omega(j-t,j-1), \; \textrm{for} \; \; 1 \leq t \leq j-1,\tag{**}\]
then the entry $(\omega(j,m)+1,j)$ is uniquely solvable.	
\end{lemma}
\begin{proof}
Let $P_{\ell, -1/q}(X) = c$ be a clue such that $(\omega(j,m)+1,j) \in \ell \cap I_{n,n}.$ Other points on $\ell \cap I_{n,n}$ are of the form $(\omega(j,m)+1 -hq, j + h)$ for some $h \in \mathbb{Z} - \{0\}.$ If $h > 0$, then $\omega(j,m)+1-hq = \omega(j,j+h-1)+1 -hq + \omega(j+h,m) \leq \omega(j+h,m)$ by (*). If $h<0$, then $\omega(j,m)+1-hq \leq \omega(j,m)+\omega(j+h,j-1) \leq \omega(j+h,m)$ by (**). Hence $(\omega(j,m)+1 -hq, j + h) \in J_{\omega}$ for $h \in \mathbb{Z} - \{0\}$. All points in $\ell \cap I_{n,n}$ are in $J_{\omega}$. Therefore the entry $(\omega(j,m)+1,j)$  is uniquely solvable as well by using the clue $P_{\ell, -1/q}(X) = c$.
\end{proof}

\begin{corollary} \label{lemma:staircasereductionadvanced}
Let $\mathscr{M} = (I_{n,n}, \mathcal{C}_T)$ be an $n \times n$ RK puzzle. Suppose that there exists a sequence of non-negative integers $\omega = (w_1, w_2, \dots, w_m)$ with $\omega(1,m) \leq n$ such that 
\begin{enumerate}
\item there exists a positive integer $m_0 \leq m$ so that the sequence $(w_1,w_2,\dots,w_{m_0})$ is monotonic non-increasing and $\{w_{m_0+1}, \dots, w_m\} \subset \{0,1\}$;
\item all entries in $J_{\omega} = \bigcup_{i=1}^m I_{\omega(i,m), i}$ are uniquely solvable. 
\end{enumerate}
For any $1 \leq j \leq m_0+1$, if there exists a clue slope $-1/q \in T$ such that
$w_j+1 \leq q \leq w_{j-1}-1$, then the entry $(\omega(j,m)+1,j)$ is uniquely solvable.
\end{corollary}
\begin{proof}
This is an easy corollary of Lemma \ref{lemma:staircasereductionbasic} and the proof is left to the reader.
\end{proof}


\begin{lemma} \label{lemma:uniquesolvregion}
For any positive integer $q \geq 2$, let $m(q):= q+(q-1)(q+2)/2.$ For every solvable $n \times n$ RK puzzle $\mathscr{M} = (I_{n,n}, \mathcal{C}_{-q})$, all entries in the set $J_{\omega_q} := \bigcup_{i=1}^{m(q)} I_{\omega_q(i,m(q)), i}$ are uniquely solvable where $\omega_q = (w^{q}_1, \dots, w^{q}_{m(q)})$ is defined as follows:
\begin{equation*}
w^q_i =
\begin{cases}
q+1-i & \text{if} \quad 1 \leq i \leq q,\\
1 & \text{if} \quad i = q+(j-1)(j+2)/2 \quad \forall \; 2 \leq j \leq q,\\
0 & \text{otherwise}.
\end{cases} 
\end{equation*}
\end{lemma}
\begin{proof}
When $q=2$, $\omega_2 = (2,1,0,1)$. It is easy to see that entries in $J_{\omega_2}$ are uniquely solvable by using clue slopes $-1$, $-1/2$ and $-2$.

Suppose that all entries in $J_{\omega_{q-1}}$ are uniquely solvable. For any $1 \leq t \leq q$ and $1 \leq s \leq q+1-t$, we define $\omega_{q-1}^{(t,s)} = (w^{(t,s)}_1, \dots, w^{(t,s)}_{m(q-1)})$ where
\begin{equation*}
w^{(t,s)}_i =
\begin{cases}
q-i & \text{if} \quad 1 \leq i \le q+1-t \text{ and } i \ne s,\\
q+1-i & \text{if} \quad q+1-t < i < q+1, \; \text{or} \; \; i=s,\\
1 & \text{if} \quad i = (q-1)+(j-1)(j+2)/2 \quad \forall \; 2 \leq j \leq q-1,\\
0 & \text{otherwise}.
\end{cases} 
\end{equation*}
We put the lexicographic order on $\{\omega_{q-1}^{(t,s)} \; | \; 1 \leq t \leq q, 1 \leq s \leq q+1-t\}$ and prove that $J_{\omega^{(t,s)}_{q-1}}$ is uniquely solvable by induction. We know that $J_{\omega_{q-1}^{(1,1)}}$ is solvable by the fact that $J_{\omega_{q-1}}$ is solvable and Corollary \ref{lemma:staircasereductionadvanced} because $w^{q-1}_1+1=q-1+1 \leq q$. Suppose that entries in $J_{\omega_{q-1}^{(t,s)}}$ are
uniquely solvable, there are two possibilities:
\begin{enumerate}
\item If $s < q+1-t$, then we want to show that $J_{\omega_{q-1}^{(t,s+1)}}$ is uniquely solvable. As $w_{s+1}^{(t,s)} +1 \leq q-s \leq w^{(t,s)}_s-1$, the entry $(\omega_{q-1}^{(t,s)}(s+1,m(q-1))+1,s+1)$ is uniquely solvable by Corollary \ref{lemma:staircasereductionadvanced}. Thus $J_{\omega_{q-1}^{(t,s+1)}}$ is uniquely solvable.
\item If $s = q+1-t$, then we want to show that $J_{\omega_{q-1}^{(t+1,1)}}$ is uniquely solvable. As $w_{1}^{(t,q+1-t)}+1 \leq q$, the entry $(\omega_{q-1}^{(t,q+1-t)}(1,m(q-1))+1,1)$ is uniquely solvable by Corollary \ref{lemma:staircasereductionadvanced}. Thus $J_{\omega_{q-1}^{(t+1,1)}}$ is uniquely solvable.
\end{enumerate}
Apply Proposition \ref{proposition:symmetrypair} to $J_{\omega_{q-1}^{(q,1)}}$, we get that entries in $J_{\omega_{q-1}^{(q+1,1)}}$ are uniquely solvable where $\omega_{q-1}^{(q+1,1)} = (w_1^{(q+1,1)}, \dots, w_{(q-1)+(q-1)(q+2)/2}^{(q+1,1)})$  is defined by:
\begin{equation*}
w^{(q+1,1)}_i =
\begin{cases}
q+1-i & \text{if} \quad 1 \leq i \leq q-1,\\
1 & \text{if} \quad i = (q-1)+(j-1)(j+2)/2 \quad \forall \; 2 \leq j \leq q,\\
0 & \text{otherwise}.
\end{cases} 
\end{equation*}
For $1 \leq s \leq q$, define $\omega_{q-1}^{(q+2,s)} = (w_1^{(q+2,s)}, \dots, w_{(q-1)+(q-1)(q+2)/2}^{(q+2,s)})$ where
\begin{equation*}
w^{(q+2,s)}_i =
\begin{cases}
q+1-i & \text{if} \quad 1 \leq i \leq q,\text{and} \; \; i \ne q-s\\
q-i & \text{if} \quad i = q-s,\\
1 & \text{if} \quad i = (q-1)+(j-1)(j+2)/2 \quad \forall \; 2 \leq j \leq q,\\
0 & \text{otherwise}.
\end{cases} 
\end{equation*}
By induction on $s$ and Corollary \ref{lemma:staircasereductionadvanced} using clue slope $-1/s$, we get that entries in $J_{\omega_{q-1}^{(q+2,s)}}$ are uniquely solvable. Apply Proposition \ref{proposition:symmetrypair} to $J_{\omega_{q-1}^{(q+2,q)}}$, we get that entries in $J_{\omega_{q}}$ are uniquely solvable.
\end{proof}

\begin{theorem} \label{theorem:maintheoremgivenclues}
Let $n \geq 1$ and $q \geq 2$ be positive integers.
\begin{enumerate}[(1)]
\item Every RK puzzle in $\textrm{RK}_{n}^{-\frac{1}{q}}$ has a unique solution if $n \leq q^2+2q-3$, or equivalently, $k_{q^2+2q-3} \preceq -\frac{1}{q}$.
\item Every RK puzzle in $\textrm{RK}_{n}^{-q}$ has a unique solution if $n \leq q^2+2q-2$, or equivalently, $k_{q^2+2q-2} \preceq -q$.
\item Every RK puzzle in $\textrm{RK}_{n}^{\frac{1}{q}}$ has a unique solution if $n \leq q^2+3q-2$, or equivalently, $k_{q^2+3q-2} \preceq \frac{1}{q}$.
\item Every RK puzzle in $\textrm{RK}_{n}^{q}$ has a unique solution if $n \leq q^2+3q-1$ or equivalently, $k_{q^2+3q-1} \preceq q$.
\end{enumerate}
\end{theorem}

\begin{proof}
We first prove (4). Given a solvable $(q^2+3q-1) \times (q^2+3q-1)$ RK puzzle $\mathscr{M}$ of clue set $\mathcal{C}_q$. Let $\omega_q$ be the sequence defined in Lemma \ref{lemma:uniquesolvregion}, we know that $J_{\omega_q}$ is solvable. Hence the first $q+\frac{(q-1)(q+2)}{2}$ entries on the first row of $\mathscr{M}$ are uniquely solvable. By Proposition \ref{proposition:symmetryvh} and the clue that gives the sum of the first row, we know that all the  $2(q+\frac{(q-1)(q)}{2})+1 = q^2+3q-1$ entries on the first row of $\mathscr{M}$ are uniquely solvable. By Corollary \ref{corollary:allequal}, we know that $k_{q^2+3q-1} \preceq q$.

The proof of (2) is similar to the proof of (4) except that we can only use Proposition \ref{proposition:symmetryvh} for $J_{\omega_{q-1}}$ and not for $J_{\omega_q}$ because we do not have clue slopes $1/q$ and $q$ available. Hence the number of uniquely solvable entries on the first row is equal to
\[\frac{q^2+3q-2}{2} + \frac{(q-1)^2+3(q-1)-2}{2}+1 = q^2+2q-2.\]
By Corollary \ref{corollary:allequal}, we know that $k_{q^2+2q-2} \preceq -q$.

Now we prove (1). Since the clue slopes are available up to $-1/q$, in the proof of Lemma \ref{lemma:uniquesolvregion}, we cannot apply Proposition \ref{proposition:symmetrypair} to show that the entries in $J_{\omega_{q-1}^{(t,s)}}$ for $t \geq q+1$ are uniquely solvable. The number of entries that are uniquely solvable on the first row of the bottom left corner is equal to the number of entries in the first row of $J_{\omega_{q-1}^{(q,1)}}$, which is $q-2+q(q+1)/2$. We can apply Proposition \ref{proposition:symmetryvh} to $J_{\omega_{q-1}}$ and the clue that gives the sum of the first row,  we get that the number of uniquely solvable entries on the first row is equal to
\[q-2+\frac{q(q+1)}{2} + \frac{(q-1)^2+3(q-1)-2}{2}+1 = q^2+2q-3.\]
By Corollary \ref{corollary:allequal}, we know that $k_{q^2+2q-3} \preceq -1/q$.

The proof of (3) is similar to the proof of (1). The total number of uniquely solvable entries on the first row is equal to
\[\frac{q^2+3q-2}{2}+q-2+\frac{q(q+1)}{2}+1 = q^3+3q-2.\]
By Corollary \ref{corollary:allequal}, we know that $k_{q^2+3q-2} \preceq 1/q$.
\end{proof}

\section{Optimality and Questions}

We do not know whether the upper bound provided in Theorem \ref{theorem:maintheoremgivenclues} is optimal or not but some computations (see the following table for the case of $n \leq 20$) have suggested that it is the case.

\begin{center}
\begin{tabular}{|c|c|c|c|}
\hline
$n$ & Slopes needed & $n$ & Slopes needed\\
\hline
$1$ & $0$ & $11$ & $0$, $\infty$, $\pm1$, $\pm1/2$, $\pm2$, $-1/3$\\
\hline
$2$ & $0$, $\infty$, $-1$ & $12$ & $0$, $\infty$, $\pm1$, $\pm1/2$, $\pm2$, $-1/3$\\
\hline
$3$ & $0$, $\infty$, $\pm1$& $13$ & $0$, $\infty$, $\pm1$, $\pm1/2$, $\pm2$, $-1/3$, $-3$\\
\hline
$4$ & $0$, $\infty$, $\pm1$, $-1/2$ & $14$ & $0$, $\infty$, $\pm1$, $\pm1/2$, $\pm2$, $\pm 1/3$, $-3$\\
\hline
$5$ & $0$, $\infty$, $\pm1$, $-1/2$ & $15$ & $0$, $\infty$, $\pm1$, $\pm1/2$, $\pm2$, $\pm 1/3$, $-3$\\
\hline
$6$ & $0$, $\infty$, $\pm 1$, $-1/2$, $-2$ & $16$ &$0$, $\infty$, $\pm1$, $\pm1/2$, $\pm2$, $\pm 1/3$, $-3$\\
\hline
$7$ & $0$, $\infty$, $\pm1$, $\pm 1/2$, $-2$ & $17$ &$0$, $\infty$, $\pm1$, $\pm1/2$, $\pm2$, $\pm 1/3$, $\pm 3$\\
\hline
$8$ & $0$, $\infty$, $\pm1$, $\pm 1/2$, $- 2$& $18$ & $0$, $\infty$, $\pm1$, $\pm1/2$, $\pm2$, $\pm1/3$, $\pm 3$, $-1/4$\\
\hline
$9$ & $0$, $\infty$, $\pm1$, $\pm1/2$, $\pm2$ & $19$ &$0$, $\infty$, $\pm1$, $\pm1/2$, $\pm2$, $\pm 1/3$, $\pm 3$, $-1/4$\\
\hline
$10$ & $0$, $\infty$, $\pm1$, $\pm1/2$, $\pm2$, $-1/3$ & $20$ & $0$, $\infty$, $\pm1$, $\pm1/2$, $\pm2$, $\pm 1/3$, $\pm 3$, $-1/4$\\
\hline
\end{tabular}
\end{center}

Further generalization can be studied if we allow clues defined by slopes of any rational numbers (instead of just coming from $S$). Also what are the optimal upper bound for the clue slopes for non-square RK puzzles?

\section*{About the author:}
 
\subsection*{Steven Rossi}

Steven Rossi was a senior at Utica College, majoring in mathematics and minoring in economics when this paper was written. After graduation, he began working for Utica National as product analyst.

Utica National Insurance Group, 180 Genesee St, New Hartford, NY 13413, strossi@utica.edu

\subsection*{Xiao Xiao}

Xiao Xiao is an associate professor of mathematics at Utica College. He received his Ph.D. from State University of New York at Binghamton in 2011. His academic interests include arithmetic geometry and inquiry-based learning.

   Department of Mathematics, Utica College, 1600 Burrstone Road, Utica, NY 13502. xixiao@utica.edu


\begin{thebibliography}{9}
     \bibitem{DAoM:GP}
         {\sc Ecke, Volker}, {\sc Fleron, Julian F.}, {\sc Hotchkiss, Philip K.} and {\sc von Renesse, Christine},
         {\em Discovering the Art of Mathematics: Game and Puzzles},
         http://www.artofmathematics.org/books/games-and-puzzles, 2015.

    \bibitem{Fleron:RKHorizons}
         {\sc Fleron, Julian F.},
         {\em {{R}adon-{K}aczmarz {P}uzzles: {CAT} {S}cans {M}eet {S}udoku}},
         Math Horizons,
         Vol. 19, Issue 3, pp28--29, Feb. 2012.

\end{thebibliography}
\end{document}